\documentclass[9pt]{article}
\usepackage {amssymb}
\usepackage {amsmath}
\usepackage {amsthm}
\usepackage{graphicx}
\usepackage {amscd}
\usepackage{color}
\usepackage[mathscr]{eucal}
\usepackage[colorlinks, citecolor=blue]{hyperref}
\usepackage[all]{xy}

\newcommand{\bx}{\mathbf{x}}
\newcommand{\ba}{\mathbf{a}}

\newcommand{\D}{{\mathcal{D}}}

\newtheorem{thm}{Theorem}
\newtheorem{prop}{Proposition}
\newtheorem{defn}{Definition}
\newtheorem{cor}{Corollary}
\newtheorem{rem}{Remark}

\newtheorem{exmp}{Example}

\newtheorem{assum}{Assumption}

\newcommand{\CC}{\mathbb{C}}

\newcommand{\NN}{\mathbb{N}}
\newcommand{\PP}{\mathbb{P}}
\newcommand{\fP}{\mathbb{P}}
\newcommand{\QQ}{\mathbb{Q}}

\newcommand{\ZZ}{\mathbb{Z}}

\def\lra{\longrightarrow}
\def\ti{\tilde}

\def\sO{{\mathcal O}}

\def\sL{{\mathcal L}}

\def\sO{\mathcal{O}}

\def\cH{\mathcal{H}}

\def\sD{\mathcal{D}}
\def\sE{\mathcal{E}}

\def\cM{\mathcal{M}}

\def\sQ{\mathscr{Q}}

\def\sX{\mathcal{X}}
\def\cX{\mathcal{X}}

\def\cZ{\mathcal{Z}}
\def\DF{{\rm DF}}
\def\bs{\mathbf{s}}
\def\ba{\mathbf{a}}
\def\bd{\mathbf{d}}
\def\cD{\mathcal{D}}
\def\fP{\mathscr{P}}
\def\Chow{\mathrm{Chow}}

\input{diagrams.sty}

\begin{document}
\title{Nonexistence of asymptotic GIT compactification}
\author{Xiaowei Wang and Chenyang Xu}
\date{\today}
\maketitle{}
\abstract{
We provide examples  of families of (log) smooth canonically polarized varieties, including smooth weighted pointed curves and smooth hypersurfaces in $\mathbb{P}^3$ with large degree  such that  the Chow semistable limits under distinct pluricanonical embeddings do not  stabilize.} \tableofcontents
\section{Introduction}

We work over the field $\mathbb{C}$ of complex numbers. Ever since the geometric invariant theory (GIT) was reinvented by Mumford, using it to construct the moduli space of polarized varieties $(X,L)$ by taking the geometric invariant theory quotient of the Chow (or Hilbert) semistable points is one of the most important applications. When $L$ is a (positive) power of the canonical bundle $\omega_X$, i.e., if $X$ is a canonically polarized manifold, it was shown by Mumford-Gieseker in dimension 1, Gieseker in dimension 2 and Donaldson-Yau in higher dimension that $(X,\omega_X)$ is asymptotically Chow stable (see \cite{MFK, G,  Yau, Don}), i.e., given a smooth canonically polarized variety $(X,\omega_X)$ we know that there exists an $r_0$ such that $(X,\omega^{\otimes r}_X)$ is Chow stable for any $r\ge r_0$.

By the general GIT theory \cite{MFK}, we indeed get a projective moduli space containing an open locus parameterizing smooth $(X,\omega^{\otimes r}_X)$.  To understand this  compactification, we need to characterize  singular objects that are included in the GIT compactification.  Since by the nature of the GIT construction the members included in this compactification a priori depend on the multiple $r$ of the pluri-canonical embedding.  A fundamental question, asked first by Koll\'ar (cf. \cite[Problem 3.1]{V}), is to see whether or not those  members in the GIT compactification will eventually stabilize, i.e., whether  Chow semistable canonically polarized  varieties  will be the same for sufficiently divisible $r$ like the case of  moduli space of curves.  However, in this note we show this fails in general even for very natural examples.

\begin{defn} A polarized log variety $(X,D;L)$ is called {\it asymptotically log Chow semistable} (resp. weakly asymptotically log Chow stable) if the corresponding Chow points of $(X,D)$ under the embeddings given by $|L^{\otimes k}|$ is semistable for  $k\gg1$ (resp. infinitely many $k\in \mathbb{N}$). (see Section \ref{sheight} for more background).
\end{defn}

When $D=0$, we will just write $(X,L)$ for $(X,D; L)$.

\begin{thm}\label{t-hyper}There exist  projective flat families  $(\sX,\sD; \sL)$ which are families of
\begin{enumerate}
\item[(1)] smooth weighted pointed curve $\pi :( \cX, \Sigma^n_{i=1} a_is_i; \omega_{\cX}(\Sigma^n_{i=1} a_is_i)) \to B^{\circ}$ for any $(a_1,...,a_n)\in \mathbb{Q}^n\cap [0,1]^n$, genus $g\ge 3$ satisfying $\sum_ia_i>\frac{g-1}{2g-4}$,
	\item[(2)] degree $m$ smooth hypersurfaces $(\cX, \sO(1))\to B^{\circ}$ in $ \mathbb{P}^3_{B^{\circ}}$ for any $m>30$, and
\item[(3)] smooth surfaces $(\mathcal{X}, \omega_{\cX})\to B^{\circ}$ whose fibers satisfy $(K^2_{\cX_b},p_g)=(2,0)$,
\end{enumerate}
 over a punctured smooth curve $B^{\circ}=B\setminus \{b_0\}$,
such that for any sufficiently divisible $r$, the Chow semistable limits of $(\mathcal{X}_b, \sD_b; \sL_b^{\otimes r})$   for $b\to b_
0$ are not weakly asymptotically log Chow semistable.
\end{thm}
The explicit families are given in Section \ref{s-ex}. From this result, the fact that we can use GIT to construct $\overline{\cM}_g$ seems to be completely a coincidence.
In fact, Theorem \ref{t-hyper} follows from the following main technical results of this paper.

\begin{thm}\label{t-main}
Let $(\cX,\sD)\to B$ be a KSBA-stable family (see Definition \ref{d-mstable}) over a smooth curve $B$. Let $r$ be a positive integer such that $r(K_{\cX}+\cD)$ is Cartier. Let $b_0\in B$ such that for any $b\in B^{\circ}=B\setminus \{ b_0\}$,  the fiber $(\cX_b,\sD_b; \omega_{\mathcal{X}}^{[r]}(rD)|_{\cX_{b}})$ is asymptotically log Chow semistable.

Then for any flat family $(\cX^{g},\sD^{g}; \sL)$ of polarized log varieties  over  $B$ satisfying
\begin{equation}\label{open}
(\cX^{g},\sD^{g}; \sL)\times_B{B^{\circ}} \cong (\cX,\sD;\omega_{\sX}^{[r]}(r\sD))\times_BB^{\circ}
\end{equation}
and with the special fiber $(\cX^{g}_{b_0},\sD^g_{b_0}; \sL^{\otimes k}|_{\sX_{b_0}})$ being log Chow semistable for infinitely many $k>0$, we have
$$(\cX^{g},\sD^{g}; \sL)\cong (\cX,\sD; \omega^{[r]}_{\sX}(\sD) ).$$
\end{thm}
As an immediate consequence, we  obtain a criterion to tell  when there is no asymptotically Chow semistable filling.

\begin{thm}\label{t-main2}
Let us continue the notation as above. If we assume the polarized pair
$$(\cX_{b_0},\cD_{b_0};\omega_{\sX}^{[r]}(r\sD)|_{\cX_{b_0}})$$
is not (weakly) asymptotically log Chow semistable. Then for any  flat family $(\cX^{g},\sD^{g}; \sL)$ satisfying (1) as above, $(\cX^{g}_{b_0},\sD^g_{b_0}; \sL|_{\sX_{b_0}})$ is not (weakly) asymptotically Chow semistable.
\end{thm}

Now we explain our strategy to prove Theorem \ref{t-main}.
Let us  recall one observation  in \cite{LX}, which says that given a family of canonically polarized varieties over a punctured curve, the KSBA-stable compactification minimizes DF invariants which are the degrees of the CM-line bundle defined in \cite{PT10} (see Section \ref{s-mmp}).
 The new input we had  in this note is that  a similar property holds for GIT stability as well.  In fact, by applying the theory  of height introduced in \cite{Wa12} (see Subsection \ref{s-height}), which plays a similar role as the degree of the CM-line bundle but for Chow line bundle, we observed that the Chow semistable filling minimizes the geometric height (see Subsection \ref{ss-hd}). However, as the limit of the normalized geometric heights is identical to the DF invariant, we immediately conclude the two special fibres for this two families must be isomorphic using the uniqueness of KSBA-stable compactification.

Finally, to prove Theorem \ref{t-main2}. We assume there is a new family with a asymptotically log Chow semistable special fiber  and satisfies \eqref{open}. Then by Theorem \ref{t-main} the central fiber must be isomorphic to $(\cX_{b_0},\cD_{b_0};\omega_{\sX}^{[r]}(r\sD)|_{\cX_{b_0}})$, contradicting to our assumption that it is not asymptotically  log Chow semistable.

In conclusion, the moral of this paper is better being illustrated as the following diagram
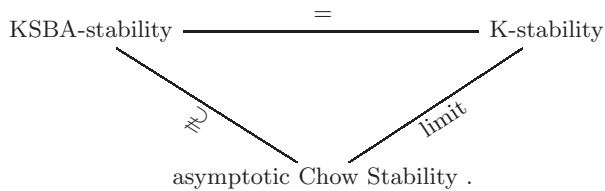
\begin{diagram}
\mbox{KSBA-stability} & & \hLine^{=} & &\mbox{K-stability}\\
& \rdLine_{\varsupsetneqq} & & \ldLine_{\mbox{limit}}\\
& & \mbox{asymptotic Chow Stability\ .}
\end{diagram}
\noindent{\bf Notations and Conventions }
See \cite{KM98} for the definitions of {\it discrepancies} and various singularities of log pairs, e.g., log canonical, klt etc.. For a proper morphism $f:Y\to X$ and a $\mathbb{Q}$-divisor $D$ on $Y$, we denote $R(Y/X,D):=\bigoplus_{m=0}f_*(mD)$.

\vskip3pt
\noindent
{\bf Acknowledgement}. The authors would like to express our gratitude to J\'anos Koll\'ar
and Baosen Wu for bring this question into our attention, also to J\'anos Koll\'ar for providing us the hypersurfaces examples in Subsection \ref{ss-sur}.
We thank Yi Hu and Yuji Odaka for their very helpful comments on the early version.  We also want to thank Prof. Tian  for point out the reference \cite{PT10}.
This collaboration started from
XW's visit to the Mathematics Department of University of Utah in September 2012. He
thanks the faculty members there for the hospitality and the research environment
they produced. CX is partially supported by grant DMS-1159175, Warnock Presidential
Endowed Chair and the Chinese government grant 'Recruitment Program of Global
Experts'.
\vskip3pt
\noindent

\section{GIT height}\label{sheight}
In this section, to make our exposition self-contained we first review the theory  of {\em GIT  height} for a family
of GIT problem as introduced in \cite{Wa12}.  In particular, we show GIT height decreases along semistable replacement. Then we will apply these constructions to study Chow stable polarized varieties. At the end, we prove that among all the birational families over $B$ that are isomorphic over an open set $U\subset B $,  over which the fibers are all Chow semistable,   a family with {\it all} fiber being Chow semistable minimizes the  GIT height.

\subsection{GIT height}\label{s-height}
To begin with,  let $( Z,\sO_Z(1)) $ be a projective variety with an action of a reductive
algebraic group $G,$ and the polarization $\sO_Z(1) $
is a $G$-linearized ample line bundle on $Z$, that is, there is a $G$-action
on the total space of $\sO_Z(1) $ that covers its
action on $Z$.

\begin{defn}\label{git}
Let $z\in Z$ and $\lambda :\mathbb{C}^*\to G$ be a one
parameter subgroup. Let
\begin{equation*}
z_{0}:=\lim_{t\to 0}\lambda ( t) \cdot z
\end{equation*}%
Then the \emph{$\lambda$-weight} of $z$, which will be denoted by $%
w_{z}( \lambda ) \in \mathbb{Z},$ is the weight of $\mathbb{C}%
^*$-action on $\sO_{Z}( 1) |_{z_{0}}\cong \mathbb{%
C}.$ A point $z$ is called \emph{(semi-)stable} with respect to the $G$%
-linearization of $\sO_{Z}( 1) $ if $w_{z}( \lambda
) $($\geq $)$>0$ for all one parameter subgroups $\lambda :\mathbb{C}%
^*\to G.$
\end{defn}

A {\em family GIT} problem over a smooth proper curve $B$ consists of the
following data:

\begin{itemize}
\item \label{ass1}Let $( Z,\sO_{Z}( 1) ) $ be
a polarized projective manifold with an action of a connected reductive
algebraic group $\tilde{G}$, where $\sO_{Z}( 1) $ is
\emph{very ample }and $\tilde{G}$-linearized. We assume the induced
representation
\begin{equation}\label{repr}
\rho :\ti{G}\lra {\rm GL}( H^{0}( Z,\sO_Z(1)))
\end{equation}
satisfies
\begin{equation}
\mathrm{Im}\rho \supset \mathbb{C}^*\cdot I,  \label{center}
\end{equation}
where $I$ is the identity in ${\rm GL}( H^{0}( Z,\sO_Z(1))) $. Let
\begin{equation}
G:=\ker ( \det \circ \rho ).  \label{G}
\end{equation}
We study the GIT problem for the $G$-action on $( Z,\sO_{Z}( 1)) .$

\item \label{ass2}Let
$$
\begin{CD}
\ti G @>>>\mathrm{Fr}^{\ti G}\\
    @.   @VV\pi V \\
    @.      B
\end{CD}
$$
be a principal $\ti G$-bundle over $B.$
\end{itemize}
Let
$
\pi _{\cZ}:\cZ:=\mathrm{Fr}^{\ti G}\times _{\ti G} Z\lra B
$
be  the fibration associated to ${\rm Fr}^{\ti G}$.
By our assumption that $\sO_{Z}( 1) $ is very ample and $\tilde{G}$-linearized, we have%
\begin{equation}
\sO_{\mathcal{Z}}( 1) :=\mathrm{Fr}^{\tilde{G}}\times _{\tilde{G}}%
\sO_{Z}( 1)  \label{O1}
\end{equation}%
is $\pi $-\emph{relative very ample. }And for a fixed $b\in B,$ there is
an isomorphism
$$\iota _{b}:( \mathcal{Z}_{b}=\pi _{\mathcal{Z}}^{-1}( b) ,\sO_{\mathcal{Z}}( 1) |_{\mathcal{Z}_{b}}) \cong ( Z,\sO_{Z}( 1) ) .$$
Moreover, there is a bundle morphism
$$
\begin{CD}
\cZ @>>>\PP\sE\\
@V\pi_{\cZ}VV @VV\pi V\\
B @= B
\end{CD}
$$
where $\sE=\pi _{\mathcal{Z}\ast }\sO_{\mathcal{Z}}( 1) =\mathrm{Fr}^{\tilde{G}}\times _{\tilde{G}}H^{0}( Z,\sO_{Z}( 1) ) \to B.$ Over $\mathcal{Z},$ there is a natural vector bundle $\pi _{\mathcal{Z}}^{\ast }\sE^{\vee }( 1) =\pi _{\mathcal{Z}}^{\ast }\sE^{\vee }\otimes \sO_{\mathcal{Z}}( 1) =\mathcal{H}om( \pi _{\mathcal{Z}}^{\ast }\sE,\sO_{\mathbb{P}\sE}( 1) |_{\mathcal{Z}}) .$

\begin{defn}
\label{ht}We define the \emph{height} $\cH_{( \cZ,\sO_{\cZ}( 1) ) }$ for the family $(
\mathcal{Z},\sO_{\mathcal{Z}}( 1) ) $ to be the
determinant line bundle $\det \pi _{\mathcal{Z}}^{\ast }\sE^{\vee
}( 1) \in {\rm Pic}( \mathcal{Z}) ,$ where ${\rm Pic}(
\mathcal{Z}) $ is the Picard variety of $\mathcal{Z}.$ For any
section $s$ of the fibration $\pi :(\cZ,\sO_\cZ(1))\to  B$ we define the \emph{height
of }$s$ to be
\begin{eqnarray*}
h_{( \mathcal{Z},\sO_{\mathcal{Z}}( 1) )
}( s)
&:=&\deg s^\ast \cH_{( \cZ,\sO_{\cZ}( 1) ) }=\deg( s^{\ast }\det \pi _{\mathcal{Z}}^{\ast }\sE^{\vee }( 1) )\\
&=&( N+1)\deg(s^{\ast }
\sO_{\mathcal{Z}}( 1) ) -\deg( \pi _{\mathcal{Z}\ast }\sO_{\mathcal{Z}}( 1) )
\end{eqnarray*}%
where $N+1$ is the rank of $\sE$.
\end{defn}

Let $s$ be a section of the fibration $\pi _{\mathcal{Z}}:( \mathcal{Z},
\sO_{\cZ}(1))\to B$.
It induces a $\tilde{G}$-equivariant $Z$-valued function of ${\rm Fr}^{\tilde{G}}$ which is, by abuse of notation, still called $s$, i.e., $s:\mathrm{Fr}^{\tilde{G}}\to  Z $ is a morphism satisfying
$s( \cdot g)=g^{-1}\cdot s( \cdot ) ,\forall g\in \tilde{G}$.
For any $k\geq 1,$ suppose that $\sigma \in H^{0}( Z,\sO_{Z}( k) ) ^{G}$ is a $G$-invariant section, then it
induces a composite morphism
\begin{equation*}
\sigma \circ s:{\rm Fr}^{\tilde{G}}\overset{s}{\longrightarrow }Z\overset{\sigma }{\longrightarrow }\sO_{Z}( k).
\end{equation*}%
By our assumption (\ref{center}), for any $g\in \tilde{G},$ there is a $h\in
\tilde{G}$ such that $\rho ( h) =\lambda I$ and $h^{-1}\cdot g\in
G.$ These imply
\begin{equation*}
\det \rho ( g) =\det \rho ( h) =\lambda ^{N+1}\text{
and }\rho _{k}( g) \cdot \sigma =\rho _{k}( h) \cdot
\sigma =\lambda ^{k}\sigma
\end{equation*}%
where $\rho _{k}:\tilde{G} \to {\rm GL}( H^{0}( Z,\sO_{Z}( k) ) ) $ is induced from the linearization of $\tilde{G}$ on $\sO_{Z}( 1) .$ Hence for any $x\in {\rm Fr}^{\tilde{G}},$
\begin{equation*}
\sigma \circ s( x\cdot g) =\sigma ( g^{-1}\cdot s(
x) ) =\rho _{k}( g) ^{-1}( \sigma \circ s)
( x) =\lambda ^{-k}\sigma \circ s( x) \in \sO_{Z}( k) |_{s( x) }\cong \mathbb{C}
\end{equation*}%
In particular, if $k=( N+1) k^{\prime }$ then $\sigma \circ
s( x\cdot g) =( \det \rho ( g) )
^{-k^{\prime }}\cdot \sigma \circ s( x) .$ So $\sigma \circ s$
defines a  section of the line bundle
$$( \det \sE) ^{\otimes ( -k^{\prime }) }\otimes s^{\ast }\sO_{\mathcal{Z}}( k) =\det ( s^{\ast }( \pi _{\mathcal{Z}}^{\ast }\sE) ) ^{\otimes ( -k^{\prime })}\otimes s^{\ast}\sO_{\mathcal{Z}}( k) $$
since $\pi _{%
\mathcal{Z}}\circ s=id,$ that is,
\begin{equation}
s^*(\sigma)=\sigma \circ s\in
H^{0}( B,s^{\ast }(
\det ( \pi _{\mathcal{Z}}^{\ast }\sE^{\vee }( 1)
) ) ^{\otimes k^{\prime }}) . \label{section}
\end{equation}%
Thus we obtain the following theorem, which refines \cite[Theorem 1.1]{CH}.

\begin{thm}
\label{CH0}Let $B$ be a smooth projective curve and
$$
\xymatrix{Z \ar@{>}[r] &\cZ \ar@{>}[d]^\pi \\
                  &B}
$$
be the fiberation we constructed above. Let $s_{0},s_{1}$ be two sections of
the fibration $\mathcal{Z}$ over $B$ and $b_{0},b_{1}\in B$ are two distinct
points

\begin{enumerate}
\item Suppose $s_{0}( b_{0}) \in \mathcal{Z}_{b_{0}}$ is
a semistable point with respect to the line bundle $\sO_{\mathcal{Z}}(
1) |_{\mathcal{Z}_{b_{0}}}$ then $h_{( \mathcal{Z},\sO_{%
\mathcal{Z}}( 1) ) }( s_{0}) \geq 0.$ Suppose
further $s_{0}( b_{1}) \in \mathcal{Z}_{b_{1}}$ is unstable with
respect to $\sO_{\mathcal{Z}}( 1) |_{\mathcal{Z}_{b_{1}}}$
then $h_{( \mathcal{Z},\sO_{\mathcal{Z}}( 1) )
}( s_{0}) >0.$

\item Suppose for any $b\in B\backslash \{ b_{0}\} $ we have $%
s_{0}( b) \in G\cdot s_{1}( b) \subset \mathcal{Z}_{b}$
and $s_{0}( b_{0}) $ is unstable but $s_{1}( b_{0}) $
is semistable. Then
\begin{equation*}
h_{( \mathcal{Z},\sO_{\mathcal{Z}}( 1) )
}( s_{0}) >h_{( \mathcal{Z},\sO_{\mathcal{Z}}(
1) ) }( s_{1}) .
\end{equation*}
\end{enumerate}
\end{thm}

\begin{proof}
It follows from GIT (cf. \cite{MFK}) that if $s_0( b) \in ( \mathcal{Z}_{b},\sO_{\mathcal{Z}}( 1) |_{\mathcal{Z}_{b}})
\cong ( Z,\sO_{Z}( 1) ) $ is semistable with
respect the $G$-action then there is a $G$-invariant section
$$\sigma \in H^{0}( Z,\sO_{Z}( k) ) ^{G}$$ for some $k\gg 1$
such that $\sigma ( s_0( b) ) \neq 0.$ Without loss of
generality, we may assume that $k=( N+1) k^{\prime }$.

By the construction above (c.f. \eqref{section}), we have
\begin{equation*}
s_0^*(\sigma )\in H^{0}( B,s_0^{\ast }( \det ( \pi ^{\ast }\sE^{\vee }( 1) ) ) ^{\otimes k^{\prime
}})
\end{equation*}%
such that $\sigma \circ s_0( b) \neq 0,$ this implies \
\begin{eqnarray*}
0 &\leq &\deg( s_0^{\ast }( \det ( \pi ^{\ast }\sE^{\vee }( 1) ) ) ^{\otimes k^{\prime }}) \\
&=&k^{\prime }( ( N+1) \deg( s_0^{\ast }\sO_{\mathcal{Z}}( 1) ) -\deg( \sE) )
\\
&=&k^{\prime }[ ( N+1) \deg( s_0^{\ast }\sO_{\mathcal{Z}}( 1) ) -\deg( \det \pi _{\ast }\sO_{\mathcal{Z}}( 1) ) ] ,
\end{eqnarray*}%
which is exactly what we want to prove. For the further part, we only need
to notice that $s_{0}( b_1) $ being unstable implies that
$\sigma ( s_{0}( b_1) ) =0$, i.e., $s_0^*(\sigma)$ is a nonvanshing section which has at least one zero.

For the second part, it follows from our assumption
that $s_{1}^*(\sigma)$ has the same zero section as $s_0^*(\sigma)$  along  $B\setminus \{b_0\}$, but at $b_0$ loses at least one zero by replacing $
s_{0}( b_{0}) $ by $s_{1}( b_{0}) $ in $\mathcal{Z}_{b_{0}}.$ Hence we obtain
\begin{equation*}
\deg ( s_{1}^{\ast }( \det ( \pi _{\mathcal{Z}}^{\ast }%
\sE^{\vee }( 1) ) ) ^{\otimes k^{\prime
}}) \leq \deg s_{0}^{\ast }( \det ( \pi _{\mathcal{Z}}^{\ast
}\sE^{\vee }( 1) ) ) ^{\otimes k^{\prime }}-1.
\end{equation*}

\end{proof}

	\begin{rem}It follows from the semistable replacement property of GIT (see {\em \cite[5.3]{Mum}}), after a possible base change of the family $\cZ\to B$, we can always decrease the {\em height} by replacing the unstable value (e.g. $s_0(b_0)$ in Theorem \ref{CH0}) by a semistable one (e.g. $s_1(b_0)$ ).
\end{rem}

\subsection{Geometric height}\label{geometric
height}
Now we apply the set-up in the previous subsection to the study a proper
flat family of polarized varieties with respect to Chow stability.  Recall for a subvariety $X\subset \PP^N$ of degree $d$, the corresponding Chow point
$$
\Chow(X)\in \PP^{d,n}:= \PP(Sym^d (\CC^{N+1} )^{\otimes(n+1)}).
$$
is the degree $(d,\cdots,d)$-hypersurface in $(\check{\PP}^N ) ^{n+1}$  consisting
of points $(H_1,\cdots, H_{n+1})\in (\check{\PP}^N ) ^{n+1}$, where $\check{\PP}^N$ is the dual of $\PP^N$, such that $H_i\subset \PP^N $ are hyperplanes satisfying $H_1\cap\cdots \cap H_{n+1}\cap X\ne \emptyset$. The group ${\rm SL}(N+1)$ naturally acts on $\PP^{d,n}$ and $\sO_{\PP^{d,n}}(1)$ is ${\rm SL}(N+1)$-linearized.  We say $X$ is Chow \emph{(semi-)stable} if
$\Chow(X)\in\PP^{d,n} $ is (semi-)stable with respect to the ${\rm SL}(N+1)$-linearization of the hyperplane bundle $\sO_{\PP^{d,n}}(1)$. Similarly, let $D=\Sigma_{i=1}^l a_iD_i$ be a divisor of $X$ with each $D_i\subset X$ being a prime divisor and $(a_1,\cdots,a_l)\in \QQ^l_{>0}$.
We let
$$
\Chow(X,D):=(\Chow(X),\Chow(D_1),\cdots,\Chow(D_l))\in \PP^{d,n}\times \prod_{i=1}^l\PP^{d_i,n-1}
$$
denote the Chow point corresponding to the log pair $(X,D)$, where $d_i=\deg(D_i)$. Let us consider the {\em diagonal} action of  ${\rm SL}(N+1)$ on $\PP^{d,n}\times \prod_{i=1}^l\PP^{d_i,n-1}$.
 We say $(X,D)$ is \emph{log Chow (semi-)stable} if
$\Chow(X,D)\in\PP^{d,n}\times \Pi_{i=1}^l\PP^{d_i,n-1} $, where $\bd:=(d,d_1,\cdots,d_l)$,
is GIT \emph{(semi-)stable} (cf. Definition \ref{git}) with respect to the ${\rm SL}(N+1)$-linearization of integral multiple of the $\QQ$-line bundle
$$\pi^{\ast}_0\sO_{\PP^{d,n}}(1)\otimes\pi^{\ast}_1\sO_{\PP^{d_1,n-1}}(c_na_1)\otimes\cdots\otimes\pi^{\ast}_l\sO_{\PP^{d_l,n-1}}(c_n a_l)
$$
covering the diagonal action of ${\rm SL}(N+1)$, where $\pi_i$ denote the projection to the $i^{\rm th}$ factor and $c_n=\frac{n+1}{2n}$. We denote
$$ \PP^{\bd,n}:=\PP^{d,n}\times \Pi_{i=1}^l\PP^{d_i,n-1}, \mbox{ where }\bd:=(d,d_1,\cdots,d_l).$$
\begin{exmp}[cf. \cite{Hass, LW}]\label{weight-pt-curve}
A {\em weighted pointed nodal curve} $(X,\ba\cdot \bx)$
consists of a reduced, connected nodal curve $X$, $l$ ordered (not necessarily distinct) {\em smooth} points
$$\bx=(x_1,\cdots,x_l)\in X^l$$ of $X$, and weights
$\ba=(a_1,\cdots,a_l)$, $a_i\in \QQ_{>0}$, of $\bx$, such that the total weight at any point is no more than one
(i.e. for any $p\in X$, $\sum_{x_i=p} a_i\leq 1$).   If in addition, $\omega_X(\ba\cdot\bx)$ is ample
then we call  $(X,\ba\cdot\bx)$ a {\em weighted pointed stable curve}.

 A polarized weighted pointed nodal curve is a weighted pointed
nodal curve together with a  very ample  polarization $L$.  Then $(X,\ba\cdot \bx; L)$ is  {\em Chow (semi-)stable} iff
$$\Chow(X,\ba\cdot\bx)\in \PP^{\bd,1}$$
 obtained from the embedding $X\subset \PP H^0(L)$  is GIT {\em (semi-)stable} with respect to the ${\rm SL}(N+1)$-linearization of integral multiple of the $\QQ$-line bundle
$$\pi^{\ast}_0\sO_{\PP^{d,1}}(1)\otimes\pi^{\ast}_1\sO_{\PP^{1,0}}(a_1)\otimes\cdots\otimes\pi^{\ast}_l\sO_{\PP^{1,0}}( a_l)
$$
since $n=1$ and $c_1=1$.
\end{exmp}

To fit the above geometric GIT construction into a family version over a proper smooth curve $B$, let us briefly summarize the construction of Chow section due to
Mumford (for details please see \cite[section 5.4]{MFK}, \cite[section 1.3]{Z} or \cite[section 4.3]{BGS}). Let $B$ be an integral scheme $\sE$ be a
locally free sheaf of rank $N+1$ over $B,$ and $\mathcal{X}$ be an effective
cycle of $\mathbb{P}\sE:=$Proj$( Sym^{\ast }\sE) $, the projective space over $B,$ whose components are flat and of relative dimension $n$ over $B$. Thus we have diagram

$$
\xymatrix{\cX\ar@{>}[r] \ar@{>}[d]_\pi    &\PP\sE \ar@{>}[d]^\pi \\
                  B\ar@{=}[r] &B}
$$

Let $\sO_{\mathbb{P}\sE}( 1) $ and $\sO_{%
\mathbb{P}\sE^{\vee }}( 1) $ denote the hyperplane line
bundle of $\mathbb{P}\sE$ and $\mathbb{P}\sE^{\vee }$
respectively. Then the canonical section of $\sE \otimes \sE^{\vee },$
which is dual to the canonical pairing $\sE^{\vee }\otimes
\sE \to \sO_{B},$ gives a section $\Delta $ of $\sO_{\mathbb{P}%
\sE}( 1) \mathcal{\boxtimes O}_{\mathbb{P}\sE%
^{\vee }}( 1) $ over $\mathbb{P}\mathcal{E\times }\mathbb{P}%
\sE^{\vee }.$ Let $\pi _{i}$ denote the $i$-th projection
\begin{equation*}
\pi _{i}:( \mathbb{P}\sE^{\vee }) ^{n+1}\longrightarrow
\mathbb{P}\sE^{\vee }
\end{equation*}%
and $\Delta _{i}$'s be the corresponding sections of $\sO_{\mathbb{P}\sE}( 1) \mathcal{\otimes \pi }_{i}^{\ast }\sO_{%
\mathbb{P}\sE^{\vee }}( 1) $ on $\mathbb{P}\sE\otimes ( \mathbb{P}\sE^{\vee }) ^{n+1}.$ Let
\begin{equation*}
\Gamma :=\bigcap\limits_{i=0}^{n}\Delta _{i}^{-1}( 0) .
\end{equation*}%
If we regard the points of $\mathbb{P}\sE^{\vee }$ as hyperplanes of$%
\mathbb{\ P}\sE$ then
\begin{equation*}
\Gamma =\{ ( x,H_{0},\cdots ,H_{n}) \in \mathbb{P}\mathcal{%
E\otimes }( \mathbb{P}\sE^{\vee }) ^{n+1}\mid \text{ \ }x\in
H_{i},\forall i\} .
\end{equation*}%
And if we regard $\Gamma $ as a correspondence from $\mathbb{P}\sE$
to $( \mathbb{P}\sE^{\vee }) ^{n+1}$%

$$
\xymatrix{& \Gamma\subset\PP\sE\otimes(\PP\sE)^{n+1}\ar@{>}[dl]_{p_1}  \ar@{>}[dr]^{p_2}& \\
 	  \PP\sE & &(\PP\sE^\vee)^{n+1}}
$$
then
\begin{equation*}
Y
:=p_{2\ast }( p_{1}^{\ast }( \mathcal{X}) \cap \Gamma
) \subset ( \mathbb{P}\sE^{\vee }) ^{n+1}
\end{equation*}%
will be a divisor of degree $( d,\cdots ,d) $ of $( \mathbb{P}\sE^{\vee }) ^{n+1}$ whose components are flat (c.f. \cite{BGS}, Lemma 4.3.1) over $B,$ where $d$ is the degree of $\mathcal{X}_{b}\subset \mathbb{P}\sE_{b}:=\pi ^{-1}( b) $ for general $b\in B.$

Let $\sO_{\fP^{d,n}}( 1) $ be the hyperplane
bundle of
\begin{equation*}
\fP^{d,n}:=\mathbb{P}[ ( Sym^{d}\sE)
^{\otimes ( n+1) }] \longrightarrow B\ .
\end{equation*}
Then the canonical pairing of $( Sym^{d}\sE) ^{\otimes
( n+1) }\otimes ( Sym^{d}\sE^{\vee })
^{\otimes ( n+1) }$ gives rise to a section $\Delta ^{\prime }$ of
the line bundle
$$\sO_{\fP^{d,n}}( 1) \otimes\mathcal{\pi }_{0}^{\ast }\sO_{\mathbb{P}\sE^{\vee }}(d) \otimes \cdots \otimes \mathcal{\pi }_{n}^{\ast }\sO_{\mathbb{P}\sE^{\vee }}( d) $$
on $\fP^{d,n}\times (\mathbb{P}\sE^{\vee }) ^{n+1}.$ By treating points of $\mathbb{P%
}[ ( Sym^{d}\sE) ^{\otimes ( n+1) }]
$ as hypersurfaces of $( \mathbb{P}\sE^{\vee }) ^{n+1}$ of
degree $( d,\cdots ,d) $, we can regard
\begin{equation*}
\Gamma ^{\prime }=\{ \Delta ^{\prime }=0\} =\{ (
H,y_{0},\cdots ,y_{n}) |( y_{0},\cdots ,y_{n}) \in
H\} .
\end{equation*}%
as a correspondence%

$$
\xymatrix{& \Gamma'\subset\fP^{d,n}\otimes(\PP\sE^\vee)^{n+1}\ar@{>}[dl]_{p_1}  \ar@{>}[dr]^{p_2}& \\
 	  \fP^{d,n} & &(\PP\sE^\vee)^{n+1}}
$$
and section
\begin{equation}
s_{(\cX,\sL)}=p_{1\ast }( p_{2}^{\ast }Y
\cap \Gamma ^{\prime })  \label{chow}
\end{equation}%
of $\fP^{d,n}$ over $B$ corresponds to $Y( \mathcal{X}%
) $ via $\Gamma ^{\prime }$ is the \emph{Chow section} for $\mathcal{X}%
.$


Let us apply the construction above to the relative Chow varieties of a family of $n$-dimensional log varieties satisfying the following
\begin{assum}\label{cond}
{\em
\label{cond} Suppose that $\cX$ is a separated scheme of dimension $%
n+1$ and $\cD=\Sigma_{i=1}^la_i\cD_i$ with $\cD_i\subset \cX$ be prime divisors. Let
\begin{equation*}
\pi :(\cX,\cD)\longrightarrow B
\end{equation*}%
be a flat proper morphism over a smooth projective curve $B$.  Let $\sL\rightarrow \cX$ be a line bundle
such that $\sE:=\pi _{\ast }\sL$  is locally free of rank $N+1 $.
Suppose that the following conditions are satisfied:
}

\begin{enumerate}
\item $\sL$ is relatively very ample and $H^0(\sX_b, \sL_b)\cong \sE_b$ for every $b\in B$;
\item There is a nonempty open set $U\subset B$ such that  $\Chow(\mathcal{X}_{b},\cD_{b})$ is a semistable point with respect to the polarization of $\sL|_{\cX_{b}}$ for all $b\in U$.
\end{enumerate}
\end{assum}
From the above assumption, we may view the pair $(\cX,\cD)$ as a family of  effective cycles
of $\PP\sE$ with $\sE=\pi _{\ast }\sL$,
$$
\xymatrix{(\cX,\cD)\ar@{>}[r] \ar@{>}[d]_\pi    &\PP\sE \ar@{>}[d]^\pi \\
                  B\ar@{=}[r] &B}
$$
whose components are flat and of dimension $n$ over $B.$
Denote by $d_0=\deg \sX_b, d_i=\deg\cD_i|_b,\ 1\leq i\leq l$, and we define
$$\fP^{\bd,n}:= \fP^{d_0,n}\times\prod_{i=1}^l\fP^{d_i,n-1} \qquad \mbox{with }\bd=(d_0,\bar\bd):=(d_0,d_1,\cdots,d_l).$$ For $\ba=(a_1,\cdots,a_l)\in (\QQ\cap[0,1])^{l}$, we introduce a $\QQ$-line bundle
\begin{equation}\label{O_C}
\sO_{\fP^{\bd,n}}(1,\ba):=\pi^{\ast}_0\sO_{\fP^{d,n}}(1)\otimes\pi^{\ast}_1\sO_{\fP^{d_1,n-1}}(c_na_1)\otimes\cdots\otimes\pi^{\ast}_l\sO_{\fP^{d_l,n-1}}(c_na_l)
\end{equation}
on $\fP^{\bd,n}$, where $\pi_i$ denote the projection of $\fP^{\bd,n}$ to the $i^{\rm th}$ factor.
  Applying Mumford's
construction of Chow section outlined above we obtain section $s_{(\cX,\cD;\sL) }$.

$$
\xymatrix{&\fP^{\bd,n}\ar@{>}[d]^\Pi \\
                 B\ar@{>}[ur]^{s_{(\cX,\cD;\sL)}} \ar@{=}[r] & B}
$$
By subsection \ref{s-height}, there is a {\em height } (cf. Definition \ref{ht}) attached to the above family of GIT problem with $G={\rm SL}(N+1)$ and $\tilde{G}={\rm GL}(N+1)$.
To unwind the underlying  meaning  of the height  in terms of the
geometry of $( \cX,\cD;\sL) ,$ we have the following:
\begin{prop}
\label{ch-ht}For the section $s_{( \cX,\cD;\sL) }$
constructed above, we have

\begin{enumerate}
\item Choosing $m\in \ZZ$ such that $c_nm\cdot\ba\in \ZZ^{l+1}$, the \emph{Chow line} is defined to be
\begin{eqnarray*}
{\rm Pic}( B)&\ni &\Lambda ^{\rm Chow}( \sL):=\cH_{(\fP^{\bd,n},\sO_{\fP^{\bd,n}}(m,m\ba))}\\
&=&\left( \det s_{(
\cX,\cD;\sL) }^{\ast }\left(\Pi^{\ast}\left( \Pi _{\ast }\sO_{%
\fP^{\bd,n}}\left(m,m\ba\right )\right)^{\vee }\otimes\sO_{\fP^{\bd,n}}\left( m,m\ba\right) \right)\right
) ^{\otimes N+1}\\
&=&\left(\left((\sL^{\langle
n+1\rangle }) ^{\otimes m}\otimes (\sL|_{\cD}^{\langle
n\rangle }) ^{\otimes c_nma}\right)^{\otimes (N+1)}\otimes( \det
\sE^{\vee }) ^{\otimes (m(n+1)d+c_nmn\ba\cdot \bar\bd)}\right ) ^{\otimes
K_{\bd,\ba,m}}
\end{eqnarray*}
where $\sL^{\langle n+1\rangle }$ and $\sL|_{\cD}^{\langle n\rangle }$  denote the Deligne pairings for the family $(\cX,\sL)\to B$ and $(\cD,\sL|_\cD)\to B$ respectively (cf. {\em\cite[Chapter 6]{De}}),  $\sE=\pi_{\ast}(\sL),\ a=\sum_ia_i$, and $$K_{\bd,\ba,m}=\dim\left(Sym^m(Sym^d(\CC^{N+1}))\oplus\bigoplus_{i=1}^lSym^{c_nma_i}(Sym^d(\CC^{N+1}))\right)\ .$$

\item We define the \emph{geometric height} to be
\begin{eqnarray*}
&&h( \cX,\cD;\sL) \\
&:=&\frac{1}{m\cdot K_{\bd,\ba,m}}h_{ \left(\fP^{\bd,n},\sO_{\fP^{\bd,n}}(m,m\ba)\right) }( s_{(
\cX,\cD;\sL) }) = \deg(\Lambda^{\rm Chow}(\sL))\\
&=&\frac{1}{( N+1)^{n}}\{( N+1)  \sL-{\rm det }\pi ^{\ast }( \pi _{\ast }\sL) \} ^{n+1}\\
& &+c_n\sum_{i=1}^l \frac{a_i}{( N+1)^{n-1}}\{ ( N+1)  \sL|_{\cD_i}-(\pi^{\cD_i}) ^{\ast }\det \pi _{\ast }\sL \} ^{n}
\end{eqnarray*}

\item  Let us define the Donaldson-Futaki invariant for the family $(\cX,\cD;\sL)\rightarrow B$ as
$$\DF((\cX,\cD)/B,\sL):= ( n+1)
\sL ^n\cdot(K_{\cX/B}+\cD)- n\sL ^{n+1}\mu
$$
with
$\mu :=\frac{L^{n-1}\cdot(K_X+D)}{L^n}$.
Then we have for $k\gg 1,$
\begin{eqnarray*}
&&h( \cX,\cD;\sL^k) \\
&=&\frac{\deg X}{2n!}\left((n+1)
\sL ^n\cdot(K_{\cX/B}+\cD)- n\sL ^{n+1}\mu \right) k^{2n}+O( k^{2n-1}) \\
&=&\frac{\deg X}{2n!}\DF(
(\cX,\cD)/B,\sL) k^{2n}+O( k^{2n-1}).
\end{eqnarray*}%

\end{enumerate}
\end{prop}
\begin{proof}It is a straightforward calculation which completely parallels to \cite[Proposition 11]{Wa12}. We omit the details.
\end{proof}

For the later application, we will need the following technical result.
\begin{prop} \label{normal}
Let $(\cX,\cD;\sL)$ be a flat family satisfying Assumption \ref{cond} and let
$\nu: \cX^\nu \to \cX$
be a finite morphism, which does not change the generic fiber over $B$. Assume $\cX^{\nu}$ is flat over $B$. Then
\begin{enumerate}
\item
There is a $k_0\in \NN$ so that for $k\geq k_0$ the new family
$(\cX^\nu, \nu^{-1}_\ast\cD; (\nu^\ast \sL)^{\otimes k})$ satisfies Assumption \ref{cond};
\item We have
$$h(\cX^\nu,\nu^{-1}_\ast\cD; (\nu^\ast \sL)^{\otimes k})\leq  h(\cX, \cD; \sL^{\otimes k})$$
with the equality holds if and only if $\cX^\nu\cong \cX$ and
$$  \DF((\cX^\nu,\nu^{-1}_\ast\cD)/B; \nu^\ast \sL)\leq  \DF((\cX, \cD)/B; \sL).$$
 \end{enumerate}
\end{prop}
\begin{proof}
The first part follows from the assumption that  $\nu$ is a finite morphism. For the second part,  if we write  $\pi^\nu=\pi\circ \nu$  then for any $k\in \NN$ we have
\begin{eqnarray*}
&&\frac{1}{(N_k+1)^n}\left\{ ( N_k+1) k( \nu^\ast\sL)
-\pi ^{\nu\ast }\det\left(\pi^\nu_\ast\nu^\ast\sL^{\otimes k}\right) \right\} ^{n+1}\\
&=&(N_k+1)k^{n+1}\sL^{n+1}- k^n(n+1) \deg X \cdot \deg \left(\det \left(\pi_\ast\nu_\ast\nu^\ast\sL^{\otimes k}\right)\right)
\\
\end{eqnarray*}
where $\deg X=[\sL|_{\cX_b}]^n$ for a general fiber $X=\cX_b$ and $N_k+1=\dim H^0(\cX_b,\sL|_{\cX_b})$, thanks to the projection formula $ \nu_*(\nu^*\sL)^{n+1}  = \sL^{n+1}$.
For the second term we consider the  exact sequence
$$
\begin{CD}
0@>>>\sO_{\cX}@>>>\nu_\ast\sO_{\cX^\nu}@>>>\sQ@>>>0.
\end{CD}
$$
Since by our assumption the support of $\sQ$ does not dominate $B$,  $\pi_\ast(\sQ\otimes_{\sO_\cX}\sL^{\otimes k})$ is a torsion sheaf on $B$. So for $k$ sufficiently large, we have $c_1(\pi_\ast(\sQ\otimes_{\sO_\cX}\sL^{\otimes k}))\geq 0$ and hence
\begin{eqnarray*}
c_1(\pi_\ast\nu_\ast(\nu^\ast\sL^{\otimes k}))
&=&c_1(\pi_*(\nu_\ast (\sO_{\cX^\nu})\otimes_{\sO_{\cX}}\sL^{\otimes k}))\\
&=&c_1(\pi_\ast(\sQ\otimes_{\sO_\cX}\sL^{\otimes k}))+c_1(\det(\pi_\ast(\sL^{\otimes k})))\\
&\ge& c_1(\pi_\ast(\sL^{\otimes k})).
\end{eqnarray*}
with the equality holds if and only if $\sQ=0$.
Putting these together, we immediately conclude that
$$h(\cX^\nu,\nu^{-1}_\ast\cD; (\nu^\ast \sL)^{\otimes k})\le h(\cX,\cD;  \sL^{\otimes k}),$$
with the equality holding if and only if $\sQ=0$. Finally, the decreasing of $\DF$ invariants under normalizaition follows immediately from part 3 of Proposition \ref{ch-ht}.
\end{proof}
\begin{rem}\label{r-normal}
This argument combined with {\em\cite[Theorem 7]{Wa12}} simplify the proof of {\em \cite[Proposition 5.1]{RT07}}. We can apply this to the normalization $\sX^{\nu}$ of $\sX$ when the generic fiber is normal, or more generally, the $S_2$-hull $\sX^{\nu}$ of $\sX$ when the generic fiber satisfies the Serre condition $S_2$.
\end{rem}
\subsection{Height decreases along semistable replacement}\label{ss-hd}
Let
$$
\xymatrix{(\cX,\cD;\sL) \ar@{-->}[rr]^f \ar@{>}[dr]  & &(\cX',\cD';\sL') \ar@{>}[dl]\\
            & B &}
$$
be two birational families satisfying Assumption \ref{cond} and
\begin{equation*}
( \cX^\circ,\cD^\circ;\sL|_{\cX^\circ})
\cong (  \cX^{\prime \circ },\cD^{\prime \circ };\sL^{\prime }|_{ \cX^{\prime \circ }}) \text{
over }B^{\circ }=B\backslash \{ b_{0}\}
\end{equation*}%
where $(\cX^{\circ },\cD^\circ)=(\cX,\cD)\times_{B}B^{\circ }$ and $
(\cX^{\prime \circ },\cD^{\prime \circ})=(\cX^{\prime },\cD^{\prime})\times_{B}B^{\circ }.$ \ Suppose there is an open $U\subset B$ containing  $%
b_{0}$ such that  $(\cX_b,\cD_b;\sL|_{\cX_b})$ is \emph{Chow semistable} for all $b\in U$ but $
(\cX^\prime_{b_0},\cD^\prime_{b_0};\sL^{\prime }|_{\cX
_{b_{0}}^{\prime }}) $ is\ \emph{Chow unstable. }By the construction
we performed in the previous subsection we obtain a section $s_{(
\cX,\cD;\sL) }$ for the fibration
\begin{equation*}
\fP^{\bd,n}\longrightarrow B.
\end{equation*}%
Now we pick two points $b_{0},b_{1}\in U$ and a section
\begin{eqnarray*}
\sigma &\in &H^{0}( \PP[ ( Sym^{d}\sE_{b_{0}}) ^{\otimes ( n+1) }]\times \Pi_{i=1}^l\PP[ ( Sym^{d_i}\sE_{b_{0}}) ^{\otimes n }] ,\sO_{\fP^{\bd,n}}(m,m\ba) |_{b_{0}}) ^{{\rm SL}( N+1) } \\
&\cong &H^{0}( \PP[ ( Sym^{d}\sE_{b_{1}}) ^{\otimes ( n+1) }]\times \Pi_{i=1}^l\PP[ ( Sym^{d_i}\sE_{b_{1}}) ^{\otimes n }] ,\sO_{\fP^{\bd,n}}(m,m\ba) |_{b_{1}}) ^{{\rm SL}( N+1) }
\end{eqnarray*}%
such that $\sigma \circ s_{(\cX,\cD;\sL)} ( b_{i}) \neq 0$ for $i=0,1,$ which can be
achieved by the semi-stability of $( \cX_{b_{i}},\cD_{b_i}; \sL|_{%
\mathcal{X}_{b_{i}}}) $ for $i=0,1.$ By
Theorem \ref{CH0}, we obtain
\begin{equation}\label{<}
0\leq h( \cX,\cD;\sL) <h( \cX^{\prime },\cD^{\prime}; \sL^{\prime }) .
\end{equation}%
Hence we have the following

\begin{prop}
\label{min}Let $( \cX,\cD;\sL) \to B$ be a family satisfying the Assumption \ref{cond}. We fix a point $b_0\in U$. Then for  any family $( \cX^{\prime },\cD^\prime;\sL^{\prime }) $ flat over $B$ satisfying the first part of Assumption \ref{cond} and
\begin{equation*}
( \cX^{\circ },\cD^\circ;\sL|_{\cX^{\circ }})
\cong (  \cX^{\prime \circ },\cD^{\prime \circ };\sL^{\prime }|_{ \cX^{\prime \circ }}) \text{
over }B^{\circ }=B\backslash \{ b_{0}\},
\end{equation*}%
we have
\begin{equation*}
0\leq h( \cX,\cD; \sL) \leq h(\cX^{\prime},\cD^{\prime}
;\sL^{\prime })
\end{equation*}%
with `$=$'~holds if and only if  $\Chow( \cX_{b_{0}}^{\prime },\cD_{b_{0}}^{\prime };
\sL^{\prime }|_{\mathcal{X}_{b_{0}}^{\prime }}) $ is $S$-equivalent to $\Chow( \cX_{b_{0}},\cD_{b_{0}}; \sL|_{\mathcal{X}_{b_{0}}})$ in $ \mathbb{P}[ ( Sym^{d}\mathbb{C}^{N+1}) ^{\otimes ( n+1) }]$.
\end{prop}

\begin{proof}
By the discussion in the beginning of this subsection, we only to prove the `$=$' case. Suppose that $\Chow( \cX_{b_{0}}^{\prime },\cD_{b_{0}}^{\prime }; \sL^{\prime }|_{\cX_{b_{0}}^{\prime }}) $ is not $S$-equivalent
to $\Chow( \cX_{b_{0}},\cD_{b_{0}};\sL|_{\mathcal{X}_{b_{0}}})$ then
it must be unstable and hence we have strict inequality \eqref{<}. Thus our proof is completed.
\end{proof}

\begin{cor}\label{c-normal}
Let $(\cX,\cD;\sL)\to B$ is a family of polarized variety. Assume $(\cX_{b_0},\cD_{b_0}; \sL^{\otimes k}|_{\cX_{b_0}})$ is Chow semistable for {\em infinitely many }$k>0$.
\begin{enumerate}
\item[(1)] If the general fibers $\cX_{b}$ are normal, then the total space $\cX\times_B U$ is normal.
\item[(2)] for any pair $( \cX^{\prime },\cD^{\prime};\sL^{\prime }) $
over $B$ satisfying
\begin{equation*}
( \cX^{\circ },\cD^{\circ}; \sL|_{\mathcal{X}^{\circ }})
\cong (  \cX^{\prime \circ },\cD^{\prime\circ};\sL%
^{\prime }|_{ \cX^{\prime \circ }}) \text{
over }B^{\circ }=B\backslash \{ b_{0}\} .
\end{equation*}
\end{enumerate}
we have $0\leq \DF( (\cX,\cD)/B,\sL) \leq \DF( (\cX^{\prime },\cD^{\prime})/B,\sL^{\prime })$.

\end{cor}
\begin{proof}
(1) It follows from Remark \ref{r-normal}.

(2) Suppose there is a pair $( \cX^{\prime },\cD^{\prime};\sL^{\prime
}) $ over $B$ satisfying%
\begin{equation*}
\DF( (\cX,\cD)/B,\sL) >\DF( (\cX^{\prime
},\cD^{\prime})/B,\sL^{\prime }) .
\end{equation*}
By the third part of Proposition \ref{ch-ht}, we deduce that for sufficiently large $k>0$, we
have
\begin{equation*}
h( \cX,\cD;\sL^{\otimes k}) >h( \cX^{\prime}, \cD^{\prime};\sL^{\prime \otimes k}),
\end{equation*}
contradicting to \ Proposition \ref{min}.
\end{proof}

\section{MMP and KSBA-stable family}\label{s-mmp}

\bigskip As we mentioned in the introduction, if we take the limit of the normalized Chow heights then what we obtain is  the degree of the CM line bundle. The corresponding notion of stability that minimizes the degree of CM line bundle is  {\it K-stability}.  In general we do not know whether K-stability yields a reasonable moduli space. However, in the canonically polarized case, there is another notion of stability called {\it KSBA stability} developed to compactify the moduli space of canonically polarized manifolds (see \cite{KSB,Al}), which indeed gives rise to a moduli space of K-stable varieties (see \cite{Od10}). In this section, we prove a family version of Odaka's result, namely, we show that the KSBA stable compactification minimizes the degree of CM line bundle among all compactifications of a KSBA stable family over a non-proper smooth curve, using the similar calculations from \cite{Od10} and \cite{LX}.

For the purpose of this note, we only consider a family over a 1-dimensional space with the generic fiber being normal. All results here can be easily generalized to the case that the generic fiber is semi-log-canonical, which is the class of singularities the fibers of a general KSBA-stable family shall be  allow to have.

\begin{defn}\label{d-mstable}
Let $(\mathcal{X},\mathcal{D})$ be a pair which is projective flat over a smooth curve $B$. We call $(\mathcal{X},\mathcal{D})$ is a {\em KSBA-stable} family over $B$  if
\begin{enumerate}
\item $K_{\mathcal{X}}+\mathcal{D}$ is $\mathbb{Q}$-Cartier,
\item  for any $b\in B$, $(\mathcal{X},\mathcal{D}+\mathcal{X}_b)$ is log-canonical, and
\item$K_{\mathcal{X}}+\mathcal{D}$ is ample over $B$.
\end{enumerate}
\end{defn}

We remark that in this note we only define the KSBA-stable family over a one parameter base. In fact, since in general for a KSBA-stable family $\mathcal{O}_{\cX}(n(K_{\cX}+\sD))\otimes \mathcal{O}_{\cX_{b}}$ may not be the same as $\mathcal{O}_{\cX_{b}}(n(K_{\cX_{b}}+\sD_b))$ for a point $b\in B$,  we avoid the subtlety of defining a KSBA-stable pair in this note.

We know that any component of $\mathcal{D}$ dominates $B$.
We first recall the following theorem proved in \cite{HX}.
\begin{thm}\label{t-uni} Let $(\mathcal{X}^{\circ},\mathcal{D}^{\circ})\to B^{\circ}$ be a flat family over $B^{\circ}=B\setminus \{b_0 \}$, where $  B$ is a smooth curve, which is a KSBA-stable family over $B^{\circ}$. Then there is a finite surjective base change $h:B_1\to B$ with a KSBA-stable family $(\mathcal{X}^s,\mathcal{D}^s)$ over a smooth curve $B_1$ such that
$$(\mathcal{X}^s,\mathcal{D}^{s})|_{h^{-1}(B^{\circ})}\cong (\mathcal{X}^{\circ},\mathcal{D}^{\circ})\times_B h^{-1}(B^{\circ}).$$ Furthermore, if there is another KSBA-stable family $(\tilde{\mathcal{X}}^s,\tilde{\mathcal{D}}^{s})$ satisfies the above condition, then $(\mathcal{X}^s,\mathcal{D}^s)\cong  (\tilde{\mathcal{X}}^s,\tilde{\mathcal{D}}^s)$ over $B_1$.
\end{thm}
\begin{proof} See \cite[Section 7]{HX}.
\end{proof}

It is easy to see that if $\phi:B_2\to B_1$ is a finite morphism, then $(\mathcal{X}^s,{\mathcal{D}}^s)\times _{B_1} B_2$ yield the unique KSBA-stable compactification of $(\mathcal{X}^{\circ},\mathcal{D}^{\circ})\times _{B_{1}}\phi^{-1}(B^0_1)$ over $B_2$.  Next we show that the above process indeed minimizes the \DF-invariant.

\begin{thm}\label{t-minK}
Let $(\sX^s,\D^s)$ be a KSBA-stable family over a smooth projective curve $B$. Let $b_0\in B$ and $B^{\circ}=B\setminus \{ b_0\}$.  Let $( \mathcal{X},\mathcal{D}; \sL) $ be
an arbitrary flat family of polarized varieties over  $B$ satisfying
$$( \mathcal{X},\D;\sL)|_{B^{\circ}} \cong ( \mathcal{X}^s, \sD^s; \omega_{\sX^s}^{[ r]}(r\sD^s))|_{B^{\circ}} $$
 for some $r>0.$ Then we have
\begin{equation*}
\DF( (\mathcal{X}^s,\cD^s)/B,\omega_{\sX^s}^{[r]}(r\sD))  \leq \DF( (\mathcal{X},\cD)/B,\sL).
\end{equation*}%
Furthermore, if  $\mathcal{X}$ is normal  then  the equality holds only if $(\mathcal{X},\sD, \sL)\cong (\mathcal{X}^s,\sD^s, \omega^{[r]}_{\mathcal{X}^s}(r\sD))$ over $B$.
\end{thm}
\begin{proof}
Let $\mathcal{X}^g$ be the graph of the rational map $\mathcal{X}\dasharrow \mathcal{X}^s$, then we have morphisms
$$
\xymatrix{& \cX^g\ar@{>}[dl]_{p_1}  \ar@{>}[dr]^{p_2}& \\
 	  \cX & &\cX^s\ .}
$$
Thus we can write $$rp_2^*(K_{\mathcal{X}^s}+\cD^s)+E=p^*_1(\sL)$$
for some divisor $E$ supported on the  special fiber.  Now we  define a function
$$f(t)=(n+1) (rp_2^*(K_{\mathcal{X}^s}+\cD^s)+tE)^{n}\cdot p_2^*(K_{\cX^s/B}+\cD)-\frac{1}{r}n(rp_2^*(K_{\mathcal{X}^s}+\cD)+tE)^{n+1},$$
thus $f(0)=\DF((\mathcal{X}^s,\D^s)/B,r(K_{\mathcal{X}^s}+\D^s))$ and
	$$\DF((\mathcal{X},\cD)/B, \sL)-f(1)=(n+1) (p^*_2\sL)^{n}\cdot F,$$
where if we denote by $\cD^g\subset\cX^g$ the birational transform of $\cD^s$, then
$$F=K_{\sX^g}+\D^g-p_2^*(K_{\sX^s}+\cD^s).$$
Since  $(\cX^s,\cD^s+\cX^s_0)$ is log canonical, and for any prime divisor $G$ whose center  on $\cX^s$ is contained in $\cX^s_0$, the discrepancies satisfy that
$$-1\le a(G,\cX^s,\cD^s+\cX^s_0)=a(G,\cX^s,\cD^s)-v_G(\cX^s_0)\le a(G,\cX^s,\cD^s)-1, $$
we have $a(G,\cX^s,\cD^s)\ge 0$. Thus $F$ is an effective divisor supported on the fiber $\sX^g$ over $0$, from which  we conclude that
$$\DF((\mathcal{X},\cD)/B, \sL)\ge f(1).$$

On the other hand, $E$ being supported on  special fiber implies
\begin{eqnarray*}
\frac{df}{dt}(t_0)= -t_0 n(n+1) (rp^*_2 (K_{\mathcal{X}^s}+\D^s)+t_0E)^{n-1}\cdot E^2\ge  0,
\end{eqnarray*}
by Hodge index theorem. As $\cX^g$ is a graph, we know that
$$rp_2^* (K_{\cX^s}+\cD)+t_0E=(1-t_0)p^*_2(K_{\cX}+\cD)+t_0p^*_1(\sL)$$ is ample over $B$ for any $t_0\in (0,1)$. So $\frac{df}{dt}(t_0)>0$ if $E\ne 0$.
Thus we conclude that
$$\DF( (\cX^s,\cD^s)/B,\sL)=f(0) \leq f(1)\le \DF(( \cX,\cD)/B,\sL).$$ with the equality holds  only if $E=0$. When $E=0$,  $rp_2^*(K_{\cX^s}+\cD^s)=p_1^*\sL$ and hence
$$\cX^{s}={\rm Proj }R(\cX^g/B, rp_2^*(K_{\cX^s}+\cD^s))\cong {\rm Proj }R(\cX^g/B, p_1^*\sL) \cong \cX,$$
where the last equality follows from the assumption that  $\cX$ is normal.\end{proof}

\begin{proof}[Proof of Theorem \ref{t-main}]
By Corollary \ref{c-normal},   $(\cX^g, \sD^g, \sL)$ minimizes the DF invariants among all compactifications of $(\sX,\sD)\times_BB^{\circ}$ over $B$ where $B^{\circ}=B\setminus \{b_0 \}$, and $\cX^g$ is normal.    Therefore, Theorem \ref{t-minK} implies $(\sX^g,\D^g; \sL)$ is indeed a KSBA-stable family over $B$. By the uniqueness (cf. Theorem \ref{t-uni}), we must have
$$(\sX^g,\sD^g;\sL)\cong(\cX_B, \sD_B; \omega_{\sX_B}^{[r]}(r\sD_B))\ .$$
Hence the proof is completed.
\end{proof}

\section{Examples}\label{s-ex}
In this section we present some examples which we can use Theorem \ref{t-main2} to show that there does not exist an asymptotically log Chow semistable limit.
\subsection{Log curves}\label{ss-lc}
In this subsection, we will discuss the examples of log curves. We will use the calculation obtained in \cite{LW}.

Let $\pi :( \cX, \ba\cdot \bs)
\to B$ be a family of \emph{weighted pointed stable curve }(cf. Example \ref{weight-pt-curve} for definitions
)\ over $B$ with $l$ sections $\bs=\{ s_{i}\} _{i=1}^{l}$ of $\cX \to B,\mathbf{a\in } [0,1] ^{l}\cap \mathbb{Q}^{l}$,  that is,  for
each $b\in B,$ $( \cX_{b},\ba\cdot\bs|_b) $ is a weighted pointed stable curve.

By definition, such a family is a KSBA-stable family over $B$.
However, whereas a smooth weighted pointed stable curve is always asymptotically Chow semistable,  there are examples of nodal curves $( X,\ba\cdot\bx) $ is {\it not}  asymptotically Chow semistable for the embedding $X\subset \PP H^0(\omega^{\otimes r}_X(r\ba\cdot \bx))$ for $r\gg 1$.

Recall from  \cite[Proposition 1.6 ]{LW}, {\em a polarized weighted pointed nodal curve} $(X, \ba\cdot\bx; L)$ with $\deg\omega_X(\ba\cdot\bx)>0$  is Chow semistable for  $\deg_L X\geq N(\ba, \deg\omega_X(\ba\cdot\bx))$, a constant only depends on $\ba$ and $\deg\omega_X(\ba\cdot\bx)$, if and only if  for any subcurve (not necessarily connected ) $Y\subset X$ we have

\begin{equation}\label{bb}
\left|\Bigl(\deg_L Y+\sum_{x_{j}\in Y}\frac{a_{j}}{2}\Bigr)-\frac{(\deg\omega_X(\ba\cdot\bx)|_Y)}{\deg\omega_X(\ba\cdot\bx)}\Bigl( \deg_L X+\sum_{j=1}^n\frac{a_{j}}{2}\Bigr)
\right|
\leq \frac{\ell_Y}{2}
\end{equation}
where $\ell_Y$ is the intersection of $Y$ with its complement in $X$. In particular, for
{\em a weighted pointed stable curve }$(X, \ba\cdot\bx; L:=\omega_X^{\otimes r}(r\ba\cdot\bx))$ is asymptotic Chow stable if and only if  for any subcurve $Y\subset X$ we have

\begin{eqnarray}\label{basic}
&&\left|\Bigl(\deg_L Y+\sum_{x_{j}\in Y}\frac{a_{j}}{2}\Bigr)-\frac{\deg(\omega_X(\ba\cdot\bx)|_Y)}{\deg\omega_X(\ba\cdot\bx)}\Bigl( \deg_L X+\sum_{j=1}^n\frac{a_{j}}{2}\Bigr)
\right|\nonumber \\
&=&\frac{1}{2}\left|\deg(\omega_X|_Y) -\frac{\deg(\omega_X(\ba\cdot \bx)|_Y)}{\deg\omega_X(\ba\cdot \bx)} \deg \omega_X
\right|\leq \frac{\ell_Y}{2}\ .
\end{eqnarray}
However, it is easy to see from \eqref{basic} that  if there is a subcurve satisfying  $\deg (\omega_X|_Y)>\ell_Y$, then a sufficiently large total weight of the  points on the complement of $Y$ will prevent $X$  being asymptotically Chow stable for the polarization $\omega_X(\ba\cdot\bx)$.

\begin{exmp}\label{P-H}
Let $X$ be a one point union of  two smooth curves $X_i,\ i=1,2$ with genus $g_1>1,g_2\geq 1$ such that $\bx$ supported on $X_2\setminus\{X_1\cap X_2\}$. Then we have $\deg(\omega_X|_{X_i})=2g_i-1, i=1,2$ and $\ell_{X_1}=1$, from which we deduce that the following inequality
\begin{eqnarray*}
&&\ell_{X_1}<\deg(\omega_X|_{X_1}) -\frac{\deg(\omega_X(\ba\cdot\bx)|_{X_1})}{\deg\omega_X(\ba\cdot\bx)} \deg \omega_X
=(2g_1-1)\left(1-\frac{g_1+g_2-1}{g_1+g_2-1+\sum_{i=1}^n a_i/2} \right)\ \\
\end{eqnarray*}
is equivalent to
\begin{equation}\label{t-wt}
\sum_ia_i>\frac{g_1+g_2-1}{2(g_1-1)}=\frac{\deg\omega_X}{2\deg\omega_{X_1}}\ .
\end{equation}
So as long as the the total weight satisfies \eqref{t-wt}, inequality \eqref{basic} will be violated and
$(X, \ba\cdot\bx; \omega_X^{\otimes r}(r\ba\cdot\bx))$ is Chow unstable for $r\gg1$.
 \end{exmp}

If we choose $( X,\ba\cdot\bx)$ to be one of the above example and put it in a family $\pi :( \cX, \ba\cdot \bs)
\to B$, such that the general fibers are smooth weighted pointed stable curves and the special fiber over $b_0$ is  $( X,\ba\cdot\bx)$, then it is a
KSBA-stable family.  However, there is no  $L$ on $X$ so that  $(X, \ba\cdot\bx; L^{\otimes r})$ is Chow semistable for $r\gg1$, since were there such a line bundle, by  letting $r\to \infty$ and applying \eqref{bb} we immediately see  this is only possible when $L$  is proportional to $\omega_X(\ba\cdot \bx)$, contradicting to the fact  that $( X,\ba\cdot\bx) $ is chosen from Example \ref{P-H}.

On the other hand, it was also proved in \cite{LW} that we can always find $\{b_i\}_i\subset \ZZ$  independent of $r$ such that polarized weighted pointed nodal curve
$$( X,\ba\cdot\bx; \omega^{\otimes r}_X( r\ba\cdot \bx) \otimes \sO_X( \Sigma_i b_{i}X_i)) $$
is asymptotically Chow semi-stable for $r\gg1$,  where $X_i$'s
are irreducible components of $X$. So for the log curve case, the pathology of nonexistence of asymptotic GIT filling was remedied  by vary the polarization while keeping  the  underlying curve $X$  unchanged. This in particular allows us to identify the GIT coarse moduli of  weighted pointed stable curve  with the one constructed by Hassett \cite{Hass}. Also see \cite[8.1]{Sw} for another GIT approach by varying the linearization weight (instead of $\ba$) on the marked points.

\subsection{Surfaces}\label{ss-sur}
Now we consider canonically polarized surfaces without boundary.
Given a family of canonically polarized smooth surfaces $\sX^{\circ}$ over $B\setminus \{b_0 \}$. They are asymptotically Chow stable (see \cite{G}). After possibly a base change  $B$, we assume  we can compactify $\sX^{\circ}$ to get a KSBA-stable family $\sX$ over $B$. In this case we know $\omega^{[r]}_{\sX}|_{\sX_{b_0}}=\omega^{[r]}_{\sX_{b_0}}$ for any $r\ge 0$, i.e., it satisfies Koll\'ar condition (see \cite{Hack}).

Let us assume  $p\in \sX_{b_0}$ is a normal point, which is a terminal point on the total family. Then it has surface singularities of class T, i.e.,   cyclic quotient singularities of type
$$\frac{1}{n^2}(1,an-1) \mbox{ with }(a,n)=1.$$
(see \cite{KSB}).

We recall that the multiplicity of a general rational singularity is $-Z^2$ where $Z$ is the fundamental cycle (cf. \cite[Corollary 6]{Ar}).
When the singularity is a cyclic quotient, the fundamental cycle $Z$  is the sum of all exceptional curves in the minimal resolution with multiplicity 1.
In general, the multiplicity of a T-singularity could be arbitrarily large.
On the other hand,  any weakly asymptotically Chow semistable polarized variety  $(X, L)$ should satisfy
${\rm mult}_x(X ) \le (\dim X + 1)! $
for any closed point $x \in X$ (see \cite[Section 3]{Mum}). Out of this, we can construction many examples. We will exhibit   two types of explicit examples in the following paragraph.

\bigskip

\noindent {\it Hypersurfaces (after Koll\'ar):} For $m> 30$, consider
$$\mathcal{X}=(w^{m-6}(xyz^4+y^6)+w^{m-10}z^{10}+t^{30}w^m+x^m+y^m+z^m=0)\in {\mathbb P}(x,y,z,w)\times \mathbb{C}[t].$$
For a general $t$, $\cX_t$ is smooth. However, $\mathcal{X}_0$ has a singularity at $(0,0,0,1)$, which is indeed not log canonical.
We take a partial resolution by performing a weighted blow up  $q:\mathcal{Z}\to \mathbb{P}^3\times \mathbb{A}^1$ at $P=(0,0,0,1)\times 0$ with weight $(1,5,6,1)$.
Let $p:\mathcal{Y}\to \cX$ be the birational transform. Then $\mathcal{Y}_0$ has two components.

We claim $\mathcal{Y}_0$ is semi-log-canonical. Locally over $P$, $\mathcal{Z}$ has a quotient singularity $Q$ of type $\frac{1}{6}(1,5,-1,1)$. Here we use $\frac{1}{r}(a_1,...,a_n)$ to denote the quotient singularity  $\mathbb{C}^n/(\ZZ/r\ZZ)$ by the action $(x_1,...,x_n)\to (\mu^{a_1}x_1,...,\mu^{a_n}x_n)$ where $\mu$ is an $n^{{\rm th}}$-primitive root of unit.  Since the changing of the variables is given by the formula
$$x=x_1z_1, y=y_1z_1^5, z=z_1^6, t=t_1z_1,$$
 locally at  $Q$, $\mathcal{Y}$ is given by
$$(x_1y_1+y_1^6+z_1^{30}+t_1^{30}+(\mbox{higher orders})=0)\subset \frac{1}{6}(1,5,-1,-1). $$
The exceptional component $E$ of $\mathcal{Y}_0$ is given by $(z_1=0)$, which is indeed normal. Denote by $(K_{\mathcal{Y}_0}+E)|_{E}=K_E+C$, then around $Q$,
$$C\subset E\subset \mathbb{P}(1,5,-1,1)=F$$
is given by the equation
$$(x_1y_1+y_1^6+(\mbox{higher orders})=0)\subset (x_1y_1+y_1^6+t_1^{30}+(\mbox{higher orders})=0)\subset \frac{1}{6}(1,5,1). $$
By inversion of adjunction, we know that $\mathcal{Y}_0$ is semi-log-canonical at $Q$. Hence $\mathcal{Y}_0$ is semi-log-canonical.

It follows from
$$K_{\mathcal{Z}}=q^*K_{\mathbb{P}^3\times \mathbb{A}_1}+12 F\qquad \mbox{and} \qquad q^*(\sX)=\mathcal{Y}+30 F$$
that $(K_Z+\mathcal{Y})|_{\mathcal{Y}}=K_\mathcal{Y}=p^*\mathcal O(m-4)-18 E.$
For any number $a>1$, $q^*\mathcal{O}(a)-F$ is ample on $\mathcal{Z}$, which implies  $K_{\mathcal{Y}}$ is ample  since $m=\deg(\mathcal{X}_t)>30>22$.
Thus $\mathcal{Y}\to \mathbb{C}[t]$ is a KSBA-stable family.

However, $Q\in E$ is a T-singularity of type
$$\frac{1}{an^2}(1,an-1)=\frac{1}{180}(1,29).$$
Its resolution is given by $(7,2,2,2,3,2,2,2,2)$, where $(b_1,b_2,...,b_n)$ denotes the dual chain of the exceptional curves with self-intersection numbers $-b_1$,..., $-b_n$. So its multiplicity is 8. Thus
${\rm mult}_Q\mathcal{Y}_0> {\rm mult}_QE=8,$ which implies $\mathcal{Y}_0$ is not asymptotically Chow semistable by Mumford's theorem.

\bigskip
\noindent{\it Surfaces with small topological invariants: } There are more complicated examples coming from the degeneration of canonically polarized surfaces with small topological invariants. In fact, in \cite{LP}, the authors  constructed a rational surface $X$ with 5 singularities and ample canonical class, which can be smoothed to a family of canonically polarized surfaces  $X_b$ with $p_g=0$ and $K^2_{X_b}=2$. The minimal resolutions of the five singularities  are given by the chains of $\PP^1$ of type
$$(4), (2,5), (2,7,2,2,3), (7,2,2,2), (2,10,2,2,2,2,2,3).$$
 These singularities have multiplicities  4, 5, 8, 7 and 11. The last three violate Mumford's inequality.

See \cite[Section 7]{LN} for more examples in this manner.

\bigskip

Department of Mathematics and Computer Sciences, Rutgers University, Newark, NJ 07102, USA

{\sl E-mail address: xiaowwan@rutgers.edu}

Beijing International Center of Mathematics Research, 5 Yiheyuan
Road, Haidian District, Beijing, 100871, China

{\sl E-mail address: cyxu@math.pku.edu.cn}

Department of Mathematics, University of Utah, 155 South 1400 East, JWB 233, Salt Lake City, UT 84112, USA

{\sl E-mail address: cyxu@math.utah.edu}

\end{document}